\documentclass[11pt,a4paper]{amsart}
\usepackage[utf8]{inputenc}
\usepackage[english]{babel}
\usepackage[T1]{fontenc}
\usepackage{amsmath}
\usepackage{amsthm}
\usepackage{amsfonts}
\usepackage{amssymb}
\allowdisplaybreaks
\usepackage{lmodern}
\usepackage{mathrsfs}
\usepackage{physics}
\usepackage[colorlinks=true,citecolor=blue,linkcolor=black]{hyperref}
\usepackage{graphicx}
\usepackage[left=2.5cm,right=2.5cm,top=2.5cm,bottom=2.5cm]{geometry}
\usepackage{enumerate}
\usepackage{tikz-cd}
\usepackage{xcolor}
\usepackage{bm}
\usetikzlibrary{backgrounds}
\usetikzlibrary{calc}
\usetikzlibrary{hobby}
\usetikzlibrary{decorations.markings}
\usetikzlibrary{arrows.meta}
\usetikzlibrary{patterns}
\usepackage{caption}
\usepackage{subcaption}
\usepackage{array}
\usepackage{float}
\usepackage{afterpage}
\usepackage{hhline}

\numberwithin{equation}{section}

\DeclareMathOperator{\Ram}{Ram}

\DeclareMathOperator{\lcm}{lcm}

\newcommand{\cir}[1]{\langle #1 \rangle}

\begin{document}

\renewcommand{\proofname}{Proof}
\renewcommand{\Re}{\operatorname{Re}}
\renewcommand{\Im}{\operatorname{Im}}
\renewcommand{\labelitemi}{$\bullet$}

\newtheorem{theorem}{Theorem}[section]
\newtheorem{proposition}[theorem]{Proposition}
\newtheorem{lemma}[theorem]{Lemma}
\newtheorem{corollary}[theorem]{Corollary}
\newtheorem{conjecture}[theorem]{Conjecture}

\theoremstyle{definition}
\newtheorem{definition}[theorem]{Definition}
\newtheorem{example}{Example}[section]
\theoremstyle{remark}
\newtheorem{remark}{Remark}[section]
\newtheorem*{claim}{Claim}


\title{New nonabelian Hodge graphs from twisted irregular connections}

\author{Jean Douçot}

\address[J.~Douçot]{`Simion Stoilow' Institute of Mathematics of the Romanian Academy,
	Calea Griviței 21,
	010702-Bucharest, 
	Sector 1, 
	Romania}
	\email{jeandoucot@gmail.com}

\maketitle

\begin{abstract}
It is known that any meromorphic connection on the Riemann sphere determines a finite diagram (a generalised graph, possibly with negative edge multiplicities), which is invariant under the Fourier-Laplace transform. If the connection has regular singularities at finite distance and an untwisted irregular singularity at infinity, the diagram is a supernova graph, and the quiver modularity theorem of Boalch/Hiroe-Yamakawa shows that the corresponding nonabelian Hodge moduli space contains the Nakajima quiver variety of the graph as an open subset.
In this note, motivated by the question of classifying these nonabelian Hodge diagrams, we obtain a necessary condition satisfied by all (core) nonabelian Hodge diagrams, involving ultrametric inequalities. We use this to investigate the class of such diagrams that are graphs (without edge-loops nor negative multiplicities), as well as certain special classes of nonabelian Hodge graphs:  we show that there exist nonabelian Hodge graphs beyond the setting of the quiver modularity theorem, and obtain characterisations of fission graphs and of simply-laced nonabelian Hodge graphs. 
\end{abstract}

\section{Introduction}

\subsection{General context and motivations}

The goal of this work is to better understand various classes of graphs and diagrams (generalised graphs) naturally appearing in the context of wild nonabelian Hodge theory and Painlevé-type equations, as some kind of Dynkin-like diagrams helping us classify the associated moduli spaces. 

It has indeed been known for some time that there are deep relations between moduli spaces of meromorphic connections on the Riemann sphere and graphs. Let us briefly recall the main facts that are relevant for our purposes here (we refer the reader to \cite{boalch2018wild, doucot2021diagrams} for more background).

\subsubsection{Nonabelian Hodge diagrams}

In the case of connections with one irregular singularity at infinity, with an \emph{untwisted} irregular class, possibly together with regular singularities at finite distance, a precise formulation of such relations has been obtained by Boalch and Hiroe-Yamakawa \cite{boalch2008irregular, boalch2012simply, hiroe2014moduli}: there is an explicit way to associate to any such connection a graph, such that there is an open subset of the corresponding moduli space that is isomorphic to the Nakajima quiver variety of the graph (viewing the graph as a doubled quiver). In turn, the full moduli space can be seen as a generalised multiplicative quiver variety associated to the graph \cite{boalch2016global}. In particular, for the moduli spaces associated to the Painlevé equations with number VI, V, IV, and II, the corresponding graphs are exactly the affine Dynkin diagrams whose Weyl groups were shown by Okamoto to be groups of symmetries of the equations \cite{okamoto1992painleve}, and in the tame case this gives the star-shaped graphs of Crawley-Boevey \cite{crawley2003matrices}. We refer to the graphs that are obtained in this way as \emph{supernova graphs}. They are obtained by gluing \emph{legs} onto a \emph{core graph}, and we refer to the possible core graphs as \emph{fission graphs}. An important case is when the graph is \emph{simply-laced} (without multiple edges), which happens when the irregular singularity is a pole of order at most 3 \cite{boalch2012simply, boalch2016global}. The fission graph is then a complete $k$-partite graph.

More recently, this picture was partially extended to the general case, allowing for \emph{twisted} irregular classes and multiple irregular singularities:  the works \cite{boalch2020diagrams, doucot2021diagrams} give an explicit construction associating a \emph{diagram}, that is a generalised graph, possibly with edge-loops and negative edge/loop multiplicities, to an arbitrary meromorphic connection on $\mathbb P^1$. This allows us in particular to associate a diagram to the moduli spaces corresponding to the remaining Painlevé equations (Painlevé III and I, and the so-called degenerate versions of Painlevé III). An important property of the construction is that the diagram is invariant under the Fourier-Laplace transform \cite{doucot2021diagrams}. We call \emph{nonabelian Hodge diagrams} the diagrams arising in this way. A nonabelian Hodge diagram may well have no loops, nor negative edge multiplicities, in that case we speak of a \emph{nonabelian Hodge graph}. As in the untwisted case, a nonabelian Hodge diagram consists of a core to which are glued some legs. 

\subsubsection{Motivation: classification of wild nonabelian Hodge moduli spaces}

Our main motivation for considering these diagrams is that they are a useful tool to classify moduli spaces of irregular connections on the Riemann sphere (see \cite{boalch2018wild} for an overview of this program). Indeed, a remarkable property of the diagrams, related to their invariance under the Fourier-Laplace transform, is that a single diagram can be obtained in several different ways, from connections with different singularity data (typically with different ranks, number of singularities, and pole orders). In the simply-laced case, these different \emph{readings} of the diagram correspond to isomorphic moduli spaces, and can essentially be determined directly from the diagram \cite{boalch2012simply, boalch2016global}. In the general case, one needs to enrich the diagram with a few extra data to reconstruct its different readings \cite{doucot2026basic}, and one also expects that they correspond to isomorphic moduli spaces. From the perspective of \emph{isomonodromic deformations} of irregular connections, which notably give rise to all Painlevé equations, the existence of such isomorphisms is closely related to the fact that Painlevé-type equations admit several \emph{Lax representations}, see for instance \cite{harnad1994dual, joshi2007linearization, joshi2009linearization, yamakawa2016fourier, bobrova2023different, alameddine2024explicit, eper2025rank_betti}.  The classification of nonabelian Hodge moduli spaces is also related to the classification of certain types of four-dimensional supersymmetric quantum field theories \cite{doucot2026fourier}, and nonabelian Hodge diagrams feature again prominently there, since they coincide in many cases with the `magnetic quivers' describing the 3d mirrors of the theories considered in the theoretical physics literature. 

Another aspect of the story is that, when the diagram is a graph, the corresponding Kac-Moody root system plays an important role: its Weyl group induces isomorphisms of moduli spaces, and an answer to the (additive) irregular Deligne-Simpson problem, asking when the moduli spaces are nonempty, can be formulated in terms of the root system \cite{boalch2008irregular, boalch2012simply,hiroe2014moduli, boalch2016global}. This is in line with the philosophy of `global Lie theory' \cite{boalch2016global}, dating back at least to \cite{laumon1988global}, where one views various objects on Riemann surfaces as global analogues of classical Lie-theoretic objects. The \emph{Cartan matrix} of the diagram (which is defined from its adjacency matrix) may be seen as a global Cartan matrix, and the nonabelian Hodge diagram as the corresponding Dynkin-type diagram.

In turn, classifying nonabelian Hodge diagrams should be an important step for making further progress towards the classification of moduli spaces in irregular 2d gauge theory. This is only one part of the work however, due to the fact that the diagrams are not a complete invariant characterising isomorphism classes of moduli spaces of irregular connections on $\mathbb P^1$: in general one needs to add some extra data to the diagram to determine the moduli space, and different diagrams may correspond to isomorphic moduli spaces. The classification of nonabelian Hodge diagrams thus has to be complemented by classifying moduli spaces corresponding to a given diagram, and finding a way to associate something like a `minimal' diagram to a given nonabelian Hodge moduli space. Such questions have been recently considered by the author in certain cases \cite{doucot2026simplification, doucot2026fourier}.

\subsection{Main results}

Classifying nonabelian Hodge diagrams in general appears to be a difficult problem, due to the fact that in the twisted case the edge/loop multiplicities depend in a quite intricate way on the types of singularities of the connections (see §\ref{sec:nonabelian_Hodge_diag}). As a first step in this direction, in the untwisted case, a count of fission graphs was recently obtained by Boalch \cite{boalch2025counting}. In this work, as a further step, we obtain a necessary condition satisfied by all nonabelian Hodge diagrams, and we bound the class of nonabelian Hodge graphs, leading in particular to characterisations of certain special subclasses.

Focusing on nonabelian Hodge \emph{graphs} is natural from point of view of root systems, since in that case the corresponding Cartan matrix is a (symmetrised) generalised Cartan matrix defining a Kac-Moody algebra. Compared to fission graphs, a natural question is how much larger the set of nonabelian Hodge graphs is, that is, how many new graphs we obtain when allowing for twisted irregular classes. Interestingly, it can be the case that a nonabelian Hodge graph defined by a connection with a twisted irregular class at infinity is actually a fission graph, i.e. it can also be obtained from an untwisted irregular class at infinity. This happens for instance for the moduli space associated to the Painlevé II equation, as observed in \cite{boalch2020diagrams}. The corresponding graph is the affine $A_1$ Dynkin diagram. It can be obtained as a nonabelian Hodge graph in several ways: from the standard Painlevé II Lax representation, corresponding to a rank 2 connection on $\mathbb P^1$ with one irregular singularity at infinity, with an untwisted irregular class corresponding to a pole of order 4; but also from the alternative `Flaschka-Newell' Lax representation, also of rank 2, having one irregular singularity at infinity with a twisted irregular class (of Poincaré rank 3/2), and a regular singularity at finite distance. Motivated by these observations, we give a positive answer to the following questions:
\begin{enumerate}
\item  Do there exist nonabelian Hodge graphs which are not fission graphs?
\item Do there exist graphs which are not nonabelian Hodge graphs? 
\end{enumerate}
We do this by finding explicit examples, such as the graphs drawn on Fig. \ref{fig:counter_examples_graphs_intro} below.

\begin{figure}[h]
\centering
\begin{subfigure}[b]{0.3\textwidth}
\centering
\begin{tikzpicture}
\tikzstyle{vertex}=[circle,fill=black,minimum size=6pt,inner sep=0pt]
\node[vertex] (A) at (0,0){} ;
\node[vertex] (B) at (2,0){} ; 
\node[vertex] (C) at (60:2){};
\draw (A) to node[midway, below] {$3$} (B);
\draw (B) to node[midway, above right] {$6$} (C);
\draw (C) to node[midway, above left] {$4$} (A);
\end{tikzpicture}
\caption{}
\label{subfig:counter_example_fission_graph_intro}
\end{subfigure}
\hfill
\begin{subfigure}[b]{0.3\textwidth}
\centering
\begin{tikzpicture}
\tikzstyle{vertex}=[circle,fill=black,minimum size=6pt,inner sep=0pt]
\node[vertex] (A) at (18:1.2){} ;
\node[vertex] (B) at (90:1.2){} ; 
\node[vertex] (C) at (162:1.2){};
\node[vertex] (D) at (234:1.2){};
\node[vertex] (E) at (306:1.2){};
\draw (A)--(B)--(C)--(D)--(E)--(A);
\draw (0,-1.3) node {};
\end{tikzpicture}
\caption{}
\label{subfig:counter_example_fission_graph_intro}
\end{subfigure}
\hfill
\begin{subfigure}[b]{0.3\textwidth}
\centering
\begin{tikzpicture}
\tikzstyle{empty}=[circle,fill=black,minimum size=0pt,inner sep=0pt]
\tikzstyle{vertex}=[circle,fill=black,minimum size=6pt,inner sep=0pt]
\node[vertex] (A) at (-1,1){};
\node[vertex] (B) at (1,1){};
\node[vertex] (C) at (1,-1){}; 
\node[vertex] (D) at (-1,-1){};
\draw (A)-- node[midway, above]{$2$} (B);
\draw (B)-- node[midway, right]{$3$} (C);
\draw (C)-- node[midway, below]{$5$} (D);
\draw (D)-- node[midway, left]{$7$} (A);
\draw (A)--  (C);
\draw (B)-- (D);
\draw  (0.6,-0.15) node {$11$};
\draw  (-0.55,-0.15) node{$13$};
\end{tikzpicture}
\caption{}
\label{subfig:counter_example_nah_graph_intro}
\end{subfigure}
\caption{(A) An example of a (core) nonabelian Hodge graph, coming from a connection with a twisted irregular class, that is not a fission graph, see Prop. \ref{prop:counter_example_fission_graph}. (B) By Thm. \ref{thm:no_new_simply_laced_graphs_intro}, the pentagon with simple edges is not a nonabelian Hodge graph. (C) A more involved example with multiple edges of a graph that is not a nonabelian Hodge graph, see Ex.~\ref{ex:other_example_not_nah_graph}. The numbers on the edges indicate their multiplicities, and no number means a multiplicity equal to 1.}
\label{fig:counter_examples_graphs_intro}
\end{figure}

We also give a simple characterisation of fission graphs, and prove that all nonabelian Hodge graphs that are simply-laced (i.e. have edge multiplicities equal to $0$ or $1$) are fission graphs, so are simple supernova graphs (as defined in \cite{boalch2008irregular, boalch2012simply}, having a complete $k$-partite core for some integer $k\geq 1$). In particular the pentagon is not a nonabelian Hodge graph, which sheds light on the fact that, while the Painlevé V equation corresponds to the square (the affine $A_3$ Dynkin diagram) and Painlevé IV corresponds to the triangle (affine $A_2$), there is no Painlevé equation corresponding to the pentagon.

Let us now give a more precise formulation of our results. To this end, we first recall and introduce some terminology:

\begin{definition}
A \emph{diagram} is a pair $(N, B)$, where $N$ is a finite set, and $B=(B_{ij})_{i,j\in N}$ is a  symmetric square matrix with integer entries $B_{ij}=B_{ji}\in \mathbb Z$ indexed by $N$, such that $B_{ii}$ is even for any $i\in N$. 
\end{definition}

We view the elements of $N$ as vertices, for distinct vertices $i,j$, we view $B_{ij}$ as a possibly negative number of edges between $i$ and $j$, and for a vertex $i$, we view $B_{ii}/2$ as a possibly negative number of loops at $i$. In brief, diagrams generalize graphs by allowing for loops and negative multiplicities. Here, by graph we will mean the following:

\begin{definition}
A \emph{graph} is a diagram $(N, B)$ without loops, nor negative edges, i.e. such that $B_{ij}\geq 0$ for all $i,j\in N$, and $B_{ii}=0$ for any $i\in N$. 
\end{definition}

The main result of \cite{boalch2020diagrams,doucot2021diagrams} is a construction which to any algebraic connection $(E, \nabla)$ on a Zariski open subset of $\Sigma=\mathbb P^1(\mathbb C)$ associates a diagram $\Gamma(E,\nabla)$, consisting of a \emph{core diagram} $\Gamma_c(E, \nabla)$, to which are glued \emph{legs} (i.e. type $A$ Dynkin diagrams). Here we will not be interested in the legs, so unless otherwise mentioned by diagram we will mean core diagram. Furthermore, since the (core) diagram is invariant under the natural action of the group $\mathrm{SL}_2(\mathbb C)$ on irregular connections on $\mathbb P^1$ including the Fourier-Laplace transform \cite{doucot2021diagrams}, and a generic element of any orbit for this action only has one singularity located at infinity, we can assume without loss of generality that the connection $(E,\nabla)$ only has a singularity at infinity, i.e. is an algebraic connection of the affine line $\Sigma^\circ:=\mathbb A^1(\mathbb C)=\mathbb P^1(\mathbb C)\smallsetminus \{\infty\}$. 

The construction is reviewed in §\ref{sec:nonabelian_Hodge_diag}. The basic idea is that the core diagram only depends on the \emph{irregular class} $\Theta$ of $(E,\nabla)$ at infinity, so we can write $\Gamma_c(E,\nabla)=\Gamma_c(\Theta)$. The vertices of $\Gamma_c(E,\nabla)$ correspond to the \emph{Stokes circles} of the connection, encoding the exponents $q$ such that exponential terms $e^q$ appear in the horizontal sections of the connection close to infinity, and the edge/loop multiplicities are determined by the interplay between the growth rates of the exponential factors terms as a function of the direction from which one approaches the singularities.

In some cases, the diagram $\Gamma_c(\Theta)$ is a graph, in particular this is always the case if $\Theta$ is \emph{untwisted}, i.e. if no root $z^{1/r}$ of the complex coordinate $z$ appears in a Stokes circle. This motivates the following definitions of various classes of diagrams/graphs:

\begin{definition}[cf. \cite{boalch2025counting}]
Let $(N, B)$ be a diagram. 
\begin{itemize}
\item $(N, B)$ is a \emph{nonabelian Hodge diagram} if there exists an algebraic connection $(E,\nabla)$ on the affine line such that $(N,B)=\Gamma_c(E,\nabla)$.
\item $(N, B)$ is a \emph{nonabelian Hodge graph} if it is a nonabelian Hodge diagram and a graph. 
\item $(N, B)$ is a \emph{fission graph} if there exists an algebraic connection $(E,\nabla)$ on the affine line $(E,\nabla)$, with  \emph{untwisted} irregular class $\Theta$ such that $(N,B)=\Gamma_c(E,\nabla)$.
\end{itemize}
\end{definition}

Clearly, we have the inclusions:

\begin{equation}
\label{eq:inclusions_types_diagrams}
\begin{tikzpicture}[scale=0.7,baseline={([yshift=-.5ex]current bounding box.center)},vertex/.style={anchor=base,
    circle,fill=black!25,minimum size=18pt,inner sep=2pt}]
\draw (-7, 0) node {$\{ \text{fission graphs}\}$};
\draw (-2,0) node {$\{ \text{NAH graphs}\}$};
\draw (0,2) node {$\{ \text{graphs}\}$};
\draw (2,0) node {$\{ \text{diagrams}\}$};
\draw (0,-2) node {$\{ \text{NAH diagrams}\}$};
\draw (-1, 1) node[rotate=45] {$\subset$};
\draw (1, 1) node[rotate=-45] {$\subset$};
\draw (-1, -1) node[rotate=-45] {$\subset$};
\draw (1, -1) node[rotate=45] {$\subset$};
\draw (-4.5,0) node{$\subset$};
\end{tikzpicture}
\end{equation}

\subsubsection{Ultrametric properties of nonabelian Hodge diagrams}

One first main result is a set of conditions satisfied by all nonabelian Hodge diagrams: the edge multiplicities in a nonabelian Hodge diagram satisfy, after a suitable rescaling in the twisted case, the defining property of an ultrametric distance on its set of vertices. Recall that a metric $d$ on a space $X$ is an ultrametric if for any triangle $\{i,j,k\}$ in $X$, its two largest sides are of equal length, i.e. if $d(i,j)\leq d(i,k)\leq d(j,k)$, then $d(i,k)=d(j,k)$.

The appearance of ultrametrics in the context of irregular connections on curves is not surprising. Indeed, it is well-known that ultrametrics are closely related to trees, and the nonabelian Hodge diagram of an irregular connection is determined by a finer invariant, the \emph{fission tree}, which encodes the topology of the successive breakings of the structure group of the connection induced by its irregular type at a singularity \cite{boalch2025twisted}. More specifically, we will see that the reason why the rescaled edge multiplicities in nonabelian Hodge diagrams satisfy the ultrametric condition is that they only depend on the closest common ancestors of pairs of leaves in the fission tree.
 
To state the condition, let us introduce some more vocabulary:

\begin{definition} A decorated diagram is a pair $(\Gamma, r)$, where $\Gamma=(N, B)$ is a diagram, and $r\in \mathbb Z_{\geq 1}^N$ is the datum of an integer $r_i\geq 1$ for each vertex $i\in N$ of the diagram. 
\end{definition}

An irregular class $\Theta$ at infinity defines a decoration $r$ of its nonabelian Hodge diagram $\Gamma_c(\Theta)$ in a canonical way, by taking the ramification orders of its Stokes circles, see §\ref{sec:nonabelian_Hodge_diag}. We will say that a decorated diagram obtained in this way from a connection on the affine line is a \emph{decorated nonabelian Hodge diagram}.

To a decorated diagram we can associate a rescaled (generalised) diagram, with rational edge/loop multiplicities:

\begin{definition}
Let $(\Gamma, r)$, with $\Gamma=(N,B)$, be a decorated diagram. The corresponding rescaled diagram is the pair $\widetilde\Gamma=(N, \widetilde B)$, where $\widetilde B=(\widetilde B_{ij})_{i,j\in N}$ is given by
\begin{equation}
\widetilde B_{ij}:=\left\lbrace 
\begin{array}{cc}
\frac{B_{ij}}{r_i r_j} & \text{ if } i\neq j,\\
\frac{B_{ii}-1}{r_i^2} & \text{otherwise.}
\end{array}
\right.
\end{equation}
\end{definition}

\begin{theorem}[Thm. \ref{thm:necessary_condition_diagram}]
\label{thm:rescaled_nah_diagram_necessary_condition_intro}
Let $(\Gamma,r)$ be a decorated nonabelian Hodge diagram, with $\Gamma=(N,B)$. Then the rescaled diagram $\widetilde \Gamma=(N, \widetilde B)$ satisfies the following properties:
\begin{itemize}
\item  Ultrametric condition for edge multiplicities: for $i,j,k\in N$ pairwise distinct, if $\widetilde B_{ij}\leq \widetilde B_{ik}\leq \widetilde B_{jk}$, then $\widetilde B_{ik}=\widetilde B_{jk}$.
\item Inequality for loops and adjacent edges: for $i,j\in N$ with $i\neq j$, $\widetilde B_{ii}\leq \widetilde B_{ij}$.
\end{itemize}
\end{theorem}

\begin{corollary}
Let $\Gamma=(N, B)$ be a nonabelian Hodge diagram. Then there exists $r\in \mathbb Z_{\geq 1}^N$ such that the rescaled diagram $\widetilde \Gamma$ associated to $(\Gamma, r)$ satisfies the conditions of Thm. \ref{thm:rescaled_nah_diagram_necessary_condition_intro}.
\end{corollary}

Interestingly, notice that Thm. \ref{thm:rescaled_nah_diagram_necessary_condition_intro} can be seen as a variant of a known result in the theory of singularities of plane algebraic curves, a theorem of Ploski \cite{ploski1985remarque} showing the existence of an ultrametric on the set of all plane branches, involving intersection numbers of pairs of branches. This is one further instance of many situations where objects associated to plane algebraic curves (Newton polygons, Puiseux series, torus knots, etc...) naturally appear in the study of irregular connections. The link between nonabelian Hodge graphs and Ploski's theorem can be sketched as follows:  via the wild nonabelian Hodge correspondence \cite{sabbah1999harmonic,biquard2004wild}, an irregular connection on the affine line determines a meromorphic Higgs bundle, which has a spectral curve. The Stokes circles of the connection are related to the branches of the spectral curve at infinity. In turn, the edge multiplicities between Stokes circles are related to intersection numbers of the corresponding branches \cite{hiroe2017ramified}. Since the ultrametric found by Ploski involves these intersection numbers, it is not surprising that this translates into a condition satisfied by nonabelian Hodge diagrams. While we prove Thm. \ref{thm:rescaled_nah_diagram_necessary_condition_intro} in a direct way, our arguments are quite similar to those used in the proof of Ploski's theorem given by Popescu-Pampu in his very nice discussion of the topic \cite{popescu2020ultrametrics}.

\subsubsection{Applications to nonabelian Hodge graphs}

We then use these conditions to bound the class of nonabelian Hodge graphs. A first result is that all the nontrivial inclusions in \eqref{eq:inclusions_types_diagrams} are strict:

\begin{theorem}[Prop. \ref{prop:counter_example_fission_graph} + Prop. \ref{prop:counter_example_nah_graph}]~
\label{thm:counter_examples_intro}
\begin{itemize}
\item There exist nonabelian Hodge graphs that are not fission graphs;
\item There exist graphs that are not nonabelian Hodge graphs (so in particular there exist diagrams that are not nonabelian Hodge diagrams).
\end{itemize}
\end{theorem}

Moreover, we obtain characterizations of some particular subclasses of nonabelian Hodge graphs. For fission graphs, the ultrametric conditions actually give a characterisation:

\begin{theorem}[Thm. \ref{thm:characterisation_fission_graphs}]
\label{thm:characterisation_fission_graphs_triangles_intro}
Let $\Gamma=(N,B)$ be a graph. Then it is a fission graph if and only if all its edge multiplicities satisfy the ultrametric condition: for $i,j,k\in N$ pairwise distinct, if $B_{ij}\leq B_{ik}\leq B_{jk}$, then $B_{ik}=B_{jk}$.
\end{theorem}

We also show that in the \emph{simply-laced} case, that is for graphs without multiple edges, we get no new such nonabelian Hodge graphs when allowing for twisted irregular classes. 

\begin{theorem}[Prop. \ref{prop:no_new_simply_laced_graphs}]
\label{thm:no_new_simply_laced_graphs_intro}
Any simply-laced (core) nonabelian Hodge diagram is a simply-laced fission graph, that is, a complete $k$-partite graph for some integer $k\geq 1$. 
\end{theorem}

For the full nonabelian Hodge graphs including legs glued to the core, this implies that any simply-laced full nonabelian Hodge graph is a simple \emph{supernova graph} as in \cite{boalch2008irregular, boalch2012simply}.

\subsection{Contents} The organization of the article is as follows. In section \ref{sec:nonabelian_Hodge_diag} we recall a few facts on nonabelian Hodge graphs and fission trees. In section \ref{sec:untwisted_case} we establish the characterisation of fission graphs and use it to find an example of a nonabelian Hodge graph that is not a fission graph. Finally in section \ref{sec:general_case} we prove Thm. \ref{thm:rescaled_nah_diagram_necessary_condition_intro} and use it to find an example of a graph that is not a nonabelian Hodge graph, and to show that all simply-laced nonabelian Hodge graphs are fission graphs. 

\subsection*{Acknowledgements} I thank Philip Boalch for suggesting to me the question that led to this work, and for many useful discussions. I am funded by the PNRR Grant CF 44/14.11.2022, ``Cohomological Hall Algebras of Smooth Surfaces and Applications''.

\section{Nonabelian Hodge diagrams}
\label{sec:nonabelian_Hodge_diag}

Let us recall a few facts about (core) nonabelian Hodge diagrams and fission trees associated to irregular connections on $\mathbb P^1$. We refer the reader to \cite{doucot2021diagrams, boalch2025twisted} for more details. As mentioned in the introduction, it follows from the invariance of the nonabelian Hodge diagrams under the Fourier-Laplace transform that we can restrict without loss of generality to the diagrams of algebraic connections on the affine line $\Sigma^\circ=\mathbb A^1(\mathbb C)=\mathbb P^1(\mathbb C)\smallsetminus \{\infty\}$. We denote by $z$ the standard complex coordinate on $\mathbb A^1(\mathbb C)$.

\subsection{General structure}

Let us first recall the notion of \emph{Stokes circles}, which naturally arises from the formulation of the formal classification of irregular connections à la Deligne-Malgrange \cite{malgrange1982classification}, and is very convenient when considering the twisted case to account for the multivaluedness of the exponential factors. 

Let $\partial$ denote the circle of real oriented directions at $\infty$. The \emph{exponential local system} $\pi:\mathcal I\to\partial$ is a covering space whose local sections on sectors correspond to germs of functions of the form
\begin{equation}
\label{eq:exponential_factor}
q=\sum_{j=1}^s a_j x^{-j/r},
\end{equation}
with $r\geq 1$, $s$ an integer, and $a_j\in \mathbb C$ for $j=1,\dots, s$, and $x:=1/z$ is the standard local coordinate on $\mathbb P^1$ vanishing at $\infty$.

By definition, a Stokes circle is a connected component of $\mathcal I$ (the connected components are homeomorphic to circles, hence the name, moreover the Stokes circles really appear as circles wrapping around the circle of directions when drawing \emph{Stokes diagrams}, see \cite{boalch2021topology}). Concretely, a Stokes circle $I$ amounts to the datum of a Galois orbit $\cir{q}$ of some principal part of a (Laurent) Puiseux series $q$ as in \eqref{eq:exponential_factor}.

If $(E,\nabla)$ is an algebraic connection on the affine line, it follows from the well-known Turrittin-Levelt theorem that $(E,\nabla)$ has a canonically defined \emph{irregular class} $\Theta$ at infinity, that is a finite multiset
\[
\Theta=\sum_{i\in N} n_i I_i,
\]
with $n_i\geq 1$ integer multiplicities, of pairwise distinct Stokes circles $I_i$.

The core nonabelian Hodge diagram $\Gamma_c(E,\nabla)=(N,B)$ associated to $(E,\nabla)$ only depends on $\Theta$. It has the following general structure:
\begin{itemize}
\item The set of vertices $N$ is identified with the set of distinct Stokes circles of $\Theta$; 
\item For $i\in N$, the loop multiplicity $B_{ii}$ is determined by the Stokes circle $I_i$;
\item For $i,j\in N$ with $i\neq j$, the edge multiplicity $B_{ij}=B_{ji}$ is determined by the pair of Stokes circles $I_i,I_j$.
\end{itemize}

In other words, any Stokes circle (at infinity) $I$ determines a loop multiplicity $B_{I,I}$ and each pair of distinct Stokes circles $(I,J)$ determines an edge multiplicity $B_{I,J}=B_{J,I}$. 

Let us now recall the explicit expressions for these multiplicities.

\subsection{Loop multiplicities}

Let $I=\cir{q}$ be a nonzero Stokes circle. There is a unique way to write the exponential factor $q$ as in \eqref{eq:exponential_factor}, with $r$ the smallest integer $r\geq 1$ such that $q$ can be written in this way, and $s$ such that $a_s\neq 0$.

The integer $r$ is the \emph{ramification order} $\mathrm{Ram}(I)=\mathrm{Ram}(q):=r$ of $q$ and $I$ (this does not depend on the Galois representative $q$). $I$ is said to be untwisted (or unramified) if $r=1$, and twisted if $r>1$. 

The integer $s$, that is the degree of $q$ as a polynomial in $z^{1/r}$ is the \emph{irregularity} $\mathrm{Irr}(I)=\mathrm{Irr}(q)$ of $q$ and $I$. 

The quotient $\frac{s}{r}=:\mathrm{Slope}(I)=\mathrm{Slope}(q)$ is the \emph{slope} of $q$ and $I$.

For $j=1, \dots, s$, we say that $\frac{j}{r}\in \mathbb Q_{>0}$ is an exponent of $I$ if $a_j\neq 0$. We denote by $E(I)\subset \frac{1}{r}\mathbb Z_{>0}$ the set of all exponents of $I$, and write them in decreasing order
\[
E(I)=\{k_0>\dots > k_p\}.
\]

We can write each exponent either as reduced fraction, or with denominator $r$:
\[
k_j=\frac{n_j}{d_j}=\frac{m_j}{r},
\]
with $n_i$ and $d_i$ coprime. This defines positive integers $n_j,d_j, m_j$ for $j=1, \dots, p$. In particular $k_0=\mathrm{Slope}(I)$ and $s=m_0$. We have $r=\lcm(d_0, \dots, d_p)$.

\begin{lemma}[{\cite[Lemma 4.6]{doucot2021diagrams}}]
\label{lemma:formula_loop_multiplicities}
With these notations, the loop multiplicity $B_{I,I}$ associated to the Stokes circle $I$ is given by
\[
B_{I,I}=m_0(r-\mathrm{gcd}(r, m_0))+ \sum_{j=1}^p  m_j\left(\mathrm{gcd}(r, m_0, \dots, m_{j-1})-\mathrm{gcd}(r, m_0,\dots,m_j)\right)-r^2+1
\]
\end{lemma}

\begin{remark} In particular, the exponent $k_j$ contributes to $B_{I,I}$ only when 
\[
\mathrm{gcd}(r, m_0,\dots,m_j)<\mathrm{gcd}(r, m_0,\dots,m_{j-1}).
\]
In that case, we say that $k_j$ is a \emph{level} of $I$. We denote by $L(I)\subset E(I)$ the set of levels of $I$. Notice that a Stokes circle $I$ is untwisted if and only if $L(I)\neq\emptyset$.

The levels of a Stokes circle essentially correspond to the notion of Puiseux characteristic of a plane branch in the theory of  singularities of plane algebraic curves. However, the notion of level is more general since it also makes sense for irregular connections on principal $G$-bundles (see \cite{doucot2025twisted} for an explicit description of levels for classical groups).
\end{remark}

\begin{remark}
In simple particular cases, the formula for the loop multiplicity simplifies:
\begin{itemize}
\item If $I=\cir{z^{s/r}}$ with $s,r$ coprime, then
\begin{equation}
B_{I,I}=(r-1)(s-r-1).
\end{equation}
\item If $I$ is untwisted, then $B_{I,I}=0$.
\end{itemize}
\end{remark}

Finally, for the tame circle $I=\cir{0}$, we have $\mathrm{Irr}(I)=0$, $\mathrm{Ram}(I)=1$, and $B_{I,I}=0$.

\subsection{Edge multiplicities}

Now consider two distinct Stokes circles $I=\cir{q}$, $J=\cir{q'}$.

We define a decomposition of $q$ and $q'$
\begin{equation}
q=q_c+q_d, \qquad q'=q'_c+q'_d,
\end{equation} into \emph{common parts} $q_c, q'_c$ such that 
$\cir{q_c}=\cir{q'_c}$, and \emph{different parts} $q_d\neq q'_d$ such that $\cir{q_d}\neq\cir{q'_d}$ in the following way: if $q=\sum_{j=1}^s a_jz^{j/r}$ is an exponential factor as in \eqref{eq:exponential_factor}, and $k\in \mathbb Q_{>0}$, we define the truncation $\tau_k(q)$ of $q$ as
\[
\tau_k(q):=\sum_{j/r\geq k} a_j z^{j/r},
\]
i.e. we remove all its monomials with exponent $<k$. The common parts are defined as 
\[
q_c:=\tau_k(q), \qquad q'_c:=\tau_k(q'), 
\]
where
\[
k=\min\left\{l\in \frac{1}{\mathrm{lcm}(r, r')}\mathbb Z_{>0}\;\middle \vert\; \cir{\tau_l(q)}=\cir{\tau_l(q')}\right\}.
\]

We say that the Stokes circle $I_c=J_c:=\cir{q_c}=\cir{q'_c}$ is the common part of $I$ and $J$ (this is well-defined since it depends only on $I$ and $J$). Note that we can have $I_c=\cir{0}$: this happens when the leading terms of $q$ and $q'$ have different Galois orbits. 

We also define the \emph{fission exponent} of $I, J$ as the rational number
\begin{equation}
f_{I, J}:=\max(\mathrm{Slope}(q_d), \mathrm{Slope}(q'_d))\in \mathbb Q_{>0}.
\end{equation}

If $I_c\neq \cir{0}$, let us write the set of exponents of the common part in decreasing order
\[
E(I_c)=\{k_0>\dots > k_t\}.
\]
We can write each exponent $k_j$, $j=0, \dots, t$ either as a reduced fraction, or with denominator $r$ or $r'$:
\[
k_j=\frac{n_j}{d_j}=\frac{m_j}{r}=\frac{m'_j}{r'},
\]
with $n_i$ and $d_i$ coprime. This defines positive integers $n_j,d_j, m_j, m'_j$ for $j=1, \dots, t$. In particular $k_0=\mathrm{Slope}(I_c)$ and $s=m_0$. The ramification order of the common part is $\Ram(I_c)=\lcm(d_0, \dots, d_t)$, and it divides both $r$ and $r'$.

\begin{lemma}[{\cite[Lemma 4.5]{doucot2021diagrams}}]
\label{lemma:formula_edge_multiplicity}
Assume $\mathrm{Slope}(q_d)\geq \mathrm{Slope}(q'_d)$, i.e. $f_{I,J}=\mathrm{Slope}(q_d)$. Then, with the previous notations and writing $f_{I,J}=\frac{s_d}{r}$, the edge multiplicity $B_{I,J}$ is given by
\begin{equation}
\begin{array}{lll}
B_{I,J} &=& m_0(r'-\mathrm{gcd}(r',m_0))+ \sum_{j=1}^t  m_j\left(\mathrm{gcd}(r', m'_0, \dots, m'_{j-1})-\mathrm{gcd}(r', m'_0,\dots,m'_j)\right)\\
& & + \; s_{d}\,\mathrm{gcd}(r', m'_0,\dots,m'_t)-rr'
\end{array}
\end{equation}
(with the convention that the first line of the right-hand side is empty if $I_c=J_c=0$).
\end{lemma}

Notice the structure of the formula: each level of the common part $I_c$ contributes one positive term (the terms coming from the other exponents of $I_c$ vanish), and there is one extra positive term coming from the fission exponent $f_{I,J}$. 

\begin{remark}
In particular, we have:
\begin{itemize}
\item If $I=\cir{\lambda z^{s/r}}$, $J=\cir{\lambda' z^{s'/r'}}$ with $\lambda, \lambda'\in\mathbb C^*$,  $s, r$ coprime, $s',r'$ coprime and $\frac{s}{r}\geq \frac{s'}{r'}$, then
\begin{equation}
B_{I,J}=r'(s-r).
\end{equation}
\item If $I$, $J$ are untwisted, i.e. the exponential factors $q$, $q'$ are polynomials in $z$, then 
\begin{equation}
B_{I,J}=\deg(q-q')-1.
\end{equation}
\end{itemize}
\end{remark}

\begin{example}~
\label{ex:some_diagrams}
\begin{itemize}
\item Consider $I=\cir{z^{5/3}}$, $J=\cir{z^{3/2}}$, $K=\cir{z^{7/3}}$. 
Since $I,J,K$ have pairwise distinct leading terms, the common parts for the pairs $(I,J)$, $(I,K)$, and $(J,K)$ are all zero. The fission exponents are 
\[
f_{I,J}=\frac{5}{3}, \quad f_{I,K}=\frac{7}{3}, \quad f_{J,K}=\frac{7}{3}.
\]
Consider the irregular class $\Theta:=I+J+K$. The matrix of edge/loop multiplicities for its Stokes circles is 
\[
B=\begin{pmatrix}
2 & 4 & 12\\
4 & 0 & 8\\
12 & 8 & 6\\
\end{pmatrix}
\]
The nonabelian Hodge diagram of $\Theta$ is thus (remember that the number of loops at $i$ in the diagram is $B_{ii}/2$):
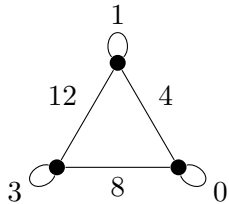
\begin{center}
\begin{tikzpicture}[scale=0.8]
\tikzstyle{vertex}=[circle,fill=black,minimum size=6pt,inner sep=0pt]
\node[vertex] (K) at (0,0){} ;
\node[vertex] (J) at (2,0){} ; 
\node[vertex] (I) at (60:2){};
\draw (I) to[out=45, in=0] ($(60:2)+(0,0.5)$) to[out=180, in=135] (I);
\draw (J) to[out=-120+45, in=-120+0] ($(J)+(-30:0.5)$) to[out=-120+180, in=-120+135] (J);
\draw (K) to[out=120+45, in=120+0] ($(K)+(-150:0.5)$) to[out=120+180, in=120+135] (K);
\draw (J) to node[midway, above right] {$4$} (I);
\draw (J) to node[midway, below] {$8$} (K);
\draw (K) to node[midway, above left] {$12$} (I);
\draw (I)++(90:0.8) node {$1$};
\draw (J)++(-30:0.8) node {$0$};
\draw (K)++(-150:0.8) node {$3$};
\end{tikzpicture}
\end{center}

\item Consider $I=\cir{z^{5/2}+z^{7/3}}$, $J=\cir{z^{5/2}+z^{5/4}}$, $K=\cir{z^{5/2}}$.
The corresponding matrix of multiplicities is 
\[
B=\begin{pmatrix}
38 & 34 & 17 \\
34 & 10 & 7 \\
17 & 7 & 2\\
\end{pmatrix}
\]
The nonabelian Hodge diagram of $\Theta$ is thus
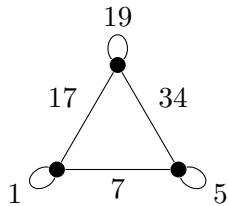
\begin{center}
\begin{tikzpicture}[scale=0.8]
\tikzstyle{vertex}=[circle,fill=black,minimum size=6pt,inner sep=0pt]
\node[vertex] (K) at (0,0){} ;
\node[vertex] (J) at (2,0){} ; 
\node[vertex] (I) at (60:2){};
\draw (I) to[out=45, in=0] ($(60:2)+(0,0.5)$) to[out=180, in=135] (I);
\draw (J) to[out=-120+45, in=-120+0] ($(J)+(-30:0.5)$) to[out=-120+180, in=-120+135] (J);
\draw (K) to[out=120+45, in=120+0] ($(K)+(-150:0.5)$) to[out=120+180, in=120+135] (K);
\draw (J) to node[midway, above right] {$34$} (I);
\draw (J) to node[midway, below] {$7$} (K);
\draw (K) to node[midway, above left] {$17$} (I);
\draw (I)++(90:0.8) node {$19$};
\draw (J)++(-30:0.8) node {$5$};
\draw (K)++(-150:0.8) node {$1$};
\end{tikzpicture}
\end{center}
\end{itemize}
\end{example}

\begin{remark}
The motivation for the definition of edge/loop multiplicities $B_{I,I}$, $B_{I,J}$ is that they correspond to the contribution to the dimension of the wild character variety $\mathcal M_B(\Theta)$ defined by the irregular class $\Theta$ of the blocks in position $I,I$ and $I,J$ in the Stokes matrices. These quantities happen to be closely related to the intersection numbers of some plane branches corresponding to the Stokes circles, see \cite{hiroe2017ramified}.
\end{remark}

\subsection{Fission trees}

The structure of these formulas for edge and loop multiplicities can be understood in a convenient way in terms of the \emph{fission tree} of the irregular class $\Theta$, a refined combinatorial object which encodes the level data of its Stokes circles as well as the fission exponents of pairs of distinct Stokes circles. 

We refer the reader to  \cite[§3.4]{boalch2025twisted} for full details of the construction of fission trees in type $A$ (fission trees have also been introduced for other classical groups \cite{doucot2026local, doucot2025twisted}), and to \cite{doucot2026basic} for many explicit examples and pictures, in the context of Painlevé equations. Here we will just recall the main properties that are useful for our purposes. 

Let $\Theta$ be an irregular class at infinity. The corresponding fission tree is a tuple $(\mathcal T,\mathbb V,\mathbb A,\mathbb L, h,n)$, where:
\begin{itemize}
\item $\mathcal T$ is a metrized tree, with set of vertices $\mathbb V$. An edge of the tree is an element of $E:=\pi_0(\mathcal T \smallsetminus V)$ of $\mathcal T$. Any vertex that is not a leaf is adjacent to $\geq 2$ edges, one of which is the parent edge, and the others are the descendant edges. The branch vertices $\mathbb Y\subset \mathbb V$ are
those with $\geq 2$ descendants. The trunk of the
tree is the union of all the edges and vertices above all the branch vertices.

If $i$ is a leaf of $\mathcal T$, we define the \emph{full branch} $\mathcal B_i$ of $i$ as the subset of $\mathcal T$ corresponding to the path from $i$ all the way to the end of the trunk.
\item $\mathbb L$ and $\mathbb A$ are two subsets of $\mathbb V$, whose elements are respectively called mandatory and admissible vertices, and $\mathbb L\subset \mathbb A$. The set $\mathbb L$ is finite and can be empty.

\item $h:\mathcal T\to \mathbb R_{\geq 0}$ is a length-preserving function, which we call the height, such that $h^{-1}(0)=:\mathbb V_0$ is the set of leaves of $\mathcal T$, $h$ maps isomorphically any edge to an interval, and for any leaf $i$, $h$ maps the full branch $\mathcal B_i$ isomorphically onto $\mathbb R_{\geq 0}$.
\item $n$ is a map $\mathbb V_0\to \mathbb Z_{\geq 1}$, i.e. consists of the datum of an integer multiplicity for each leaf. 
\end{itemize}

Moreover, the fission tree has the following properties:
\begin{itemize}
\item The leaves are identified with the Stokes circles of $\Theta$: if $\Theta=\sum_{i\in N}n_i I_i$, we have $\mathbb V_0=N$, the leaf $i$ corresponds to the Stokes circle $I_i$, and its multiplicity is $n_i$.
\item If $i\in \mathbb V_0$ is a leaf and $I$ is the corresponding Stokes circle, the heights of the mandatory vertices of the full branch $\mathcal B_i$ are exactly the levels of $I$, that is
\[
h(\mathbb L\cap \mathcal B_i)=L(I),
\]
and all the exponents of $I$ are heights of admissible vertices of  $\mathcal B_i$, i.e. 
\[
h(\mathbb A\cap \mathcal B_i)\supset E(I).
\]
\item If $i,j\in \mathbb V_0$ are two distinct leaves and $I,J$ are the corresponding Stokes circles of $\Theta$, and $v_{ij}\in \mathbb V$ is the closest common ancestor of $i,j$ (in particular $v\in \mathbb Y$, i.e. $v_{ij}$ is a branching vertex), then:
\begin{itemize}
 \item All exponents (and in particular all levels) of the common part $I_c=J_c$ of $I$ and $J$ are greater than the height $h(v_{ij})$, i.e. 
 \[
 \min(E(I_c))\geq h(v_{ij})
 \]
 \item The height of all children of $v_{ij}$ is equal to the fission exponent $f_{I,J}$.
\end{itemize}
\end{itemize}

\begin{example} As in Example \ref{ex:some_diagrams}, consider $I=\cir{z^{5/3}}$, $J=\cir{z^{3/2}}$, $K=\cir{z^{7/3}}$, and the irregular class $\Theta:=I+J+K$. Its fission tree is drawn below: 
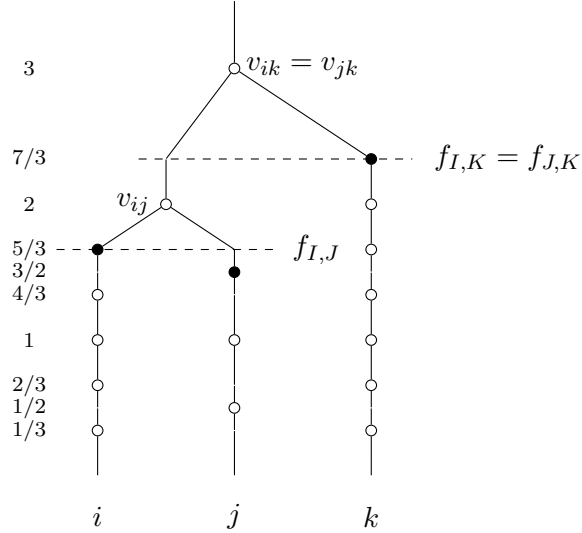
\begin{center}
\begin{tikzpicture}[scale=1.8]
\tikzstyle{empty}=[circle,fill=black,minimum size=0pt,inner sep=0pt]
\tikzstyle{mandatory}=[circle,fill=black,minimum size=4pt,draw, inner sep=0pt]
\tikzstyle{inc}=[circle,fill=white,minimum size=4pt,draw, inner sep=0pt]
\tikzstyle{indeterminate}=[circle,densely dotted,fill=white,minimum size=4pt,draw, inner sep=0pt]
\node[empty] (R) at (2,3.5){};
\node[inc] (X) at (2,3){};
\node[mandatory] (i7/3) at (3, 7/3){};
\node[empty] (jk7/3) at (1.5, 7/3){};
\node[inc] (i2) at (3, 2){};
\node[inc] (jk2) at (1.5, 2){};

\node[inc] (i5/3) at (3, 5/3){};
\node[mandatory] (j5/3) at (1, 5/3){};
\node[empty] (k5/3) at (2, 5/3){};

\node[empty] (i3/2) at (3, 3/2){};
\node[empty] (j3/2) at (1, 3/2){};
\node[mandatory] (k3/2) at (2, 3/2){};

\node[inc] (i4/3) at (3, 4/3){};
\node[inc] (j4/3) at (1, 4/3){};
\node[empty] (k4/3) at (2, 4/3){};

\node[inc] (i1) at (3, 1){};
\node[inc] (j1) at (1, 1){};
\node[inc] (k1) at (2, 1){};

\node[inc] (i2/3) at (3, 2/3){};
\node[inc] (j2/3) at (1, 2/3){};
\node[empty] (k2/3) at (2, 2/3){};

\node[empty] (i1/2) at (3, 1/2){};
\node[empty] (j1/2) at (1, 1/2){};
\node[inc] (k1/2) at (2, 1/2){};

\node[inc] (i1/3) at (3, 1/3){};
\node[inc] (j1/3) at (1, 1/3){};
\node[empty] (k1/3) at (2, 1/3){};

\node[empty] (i) at (3, 0){};
\node[empty] (j) at (1, 0){};
\node[empty] (k) at (2, 0){};

\draw (1,-0.3) node {$i$};
\draw (2,-0.3) node {$j$};
\draw (3,-0.3) node {$k$};

\foreach \n / \h in {3/3,{7/3}/{7/3},2/2, {5/3}/{5/3}, {3/2}/{3/2}, {4/3}/{4/3}, 1/1, {2/3}/{2/3}, {1/2}/{1/2}, {1/3}/{1/3}} 
\draw (0.5, \h) node{{\scriptsize $\n$}}; 
\draw (R)--(X);
\draw (X)--(jk7/3)--(jk2);

\draw (X)--(i7/3)--(i2)--(i5/3)--(i3/2)--(i4/3)--(i1)--(i2/3)--(i1/2)--(i1/3)--(i);
\draw(jk2)--(j5/3)--(j3/2)--(j4/3)--(j1)--(j2/3)--(j1/2)--(j1/3)--(j);
\draw(jk2)--(k5/3)--(k3/2)--(k4/3)--(k1)--(k2/3)--(k1/2)--(k1/3)--(k);

\draw[dashed] (0.7,5/3)--(2.3,5/3);
\draw (2.6,5/3) node {$f_{I,J}$};
\draw[dashed] (1.3,7/3)--(3.3,7/3);
\draw (4,7/3) node {$f_{I,K}=f_{J,K}$};

\draw (1.26,2) node {$v_{ij}$};
\draw (2.5,3) node {$v_{ik}=v_{jk}$};
\end{tikzpicture}
\end{center}
In this picture, the leaves $i,j,k$ of the tree correspond to the Stokes circles $I,J,K$ respectively. The rational numbers on the left are the heights of the vertices. The mandatory vertices are drawn in black, the other admissible vertices are drawn in white, and the remaining \emph{empty} vertices are not drawn, in agreement with the conventions used in \cite{boalch2025twisted}.
\end{example}

\section{Untwisted case and new nonabelian Hodge graphs}
\label{sec:untwisted_case}

In this section, we give a simple characterisation of fission graphs, and use it to show that there exist nonabelian Hodge graphs coming from twisted irregular classes that are not fission graphs.

\begin{definition}
Let $\Gamma=(N,B)$ be a graph, and $T=\{i,j,k\}\subset N$ a triangle of $\Gamma$. We say that $T$ is \emph{acute isosceles} if the two largest numbers among its edge multiplicities $B_{ij}$, $B_{jk}$, $B_{ik}$ are equal.
\end{definition}

\begin{theorem}\label{thm:characterisation_fission_graphs}
Let $\Gamma$ be a graph. Then it is a fission graph if and only if all its triangles are acute isosceles. 
\end{theorem}

\begin{proof}
Let $\Gamma=(N, B)$ be a fission graph. By definition, there exists an untwisted irregular class $\Theta$ (unique up to admissible deformations) such that $\Gamma=\Gamma_c(\Theta)$. The set $N$  of vertices of $\Gamma$ is the set of Stokes circles of $\Theta$. 
Let $i, j, k$ be three vertices of $\Gamma$, and let $\langle q_i\rangle$, $\langle q_j \rangle$, $\langle q_k \rangle$ be the corresponding Stokes circles. By definition of fission graphs, we have $B_{ij}=\deg(q_i-q_j)-1$, $B_{jk}=\deg(q_j-q_k)-1$, $B_{ik}=\deg(q_i-q_k)-1$. Without loss of generality we may assume $B_{ij}\leq B_{ik}\leq B_{jk}$. Now either $B_{ij}=B_{ik}=B_{jk}$, and in that case the triangle $ijk$ is ``equilateral'', and in particular acute isosceles, or $B_{ij}<B_{jk}$. In the latter case, we have $\deg(q_i-q_j)<\deg(q_j-q_k)$, which implies $\deg(q_i-q_k)=\deg((q_i-q_j)+(q_j-q_k))=\deg(q_j-q_k)$, so $B_{ik}=B_{jk}$ and the triangle $ijk$ is acute isosceles.

Conversely, if $\Gamma=(N, B)$ is a graph such that all its triangles are acute isosceles, we can construct from it an untwisted fission tree. We proceed as follows.  Let $K-1$ be the maximal edge multiplicity appearing in $\Gamma$. We construct recursively, for $h=0, \dots, K$, an untwisted fission forest (i.e. a finite collection of untwisted fission trees) $\mathbf F_h$ with the following properties $(P_h)$. 
\begin{itemize}
\item The set of leaves of $\mathbf F_h$ is $N_0:=N$. 
\item The set of vertices of $\mathbf F_h$ is a disjoint union $N_0\sqcup N_1 \sqcup \dots \sqcup N_h$. If $x\in N_l$, with $l=0, \dots, h$, is a vertex of $\mathbf F_h$ we say that $x$ is of height $l$. 
\item For $l=1,\dots, h$, if $x, y\in N_l$, $x\neq y$ are two distinct vertices of height $l$, if $i\in N$ is any leaf that is a descendant of $x$ and $j\in N$ is any leaf that is a descendant of $y$, we have $B_{ij}\leq l-1$. Furthermore, if $x,y\in N$ are two leaves, and $1\leq l\leq h$, then $B_{xy}=l-1$ if and only if $x$ and $y$ have a common ancestor of height $l$ in $\mathbf F_h$. 
\end{itemize}
The forests $\mathbf F_h$ are defined as follows:

\begin{itemize}
\item $\mathbf F_0$ has set of vertices $N_0$ and no edges. 
\item Let us define $\mathbf F_1$. If $i,j\in N_0$, we write $i\sim_0 j$ if $B_{ij}=0$. This yields a well-defined equivalence relation on $N_0$, indeed it is reflexive and symmetric, and for the transitivity if $i, j,k\in N_0$ are pairwise distinct and such that $B_{ij}=B_{jk}=0$, then since the triangle $ijk$ is acute isosceles we also have $B_{ik}=0$. So we define $N_1$ to be the set of equivalence classes of $N_0$ for the relation $\sim_0$. The edges are defined as follows: for each $i\in N_0$ we draw an edge between $i$ and the vertex of $N_1$ representing its equivalence class. By construction, clearly $\mathbf F_1$ satisfies the property $(P_1)$

\item For $h\leq k$, assuming $\mathbf F_0, \dots, \mathbf F_{h-1}$ already defined and satisfying the sought after properties, let us define $\mathbf F_h$. If $i, j\in N_{h-1}$, we say that $i\sim_{h-1} j$ if for some (or equivalently any) leaf $x\in N_0$ descendant of $i$, and some (or equivalently any) leaf $y\in N_0$ descendant of $j$, we have $B_{xy}=h-1$. We claim that this yields a well-defined equivalence relation $\sim_{h-1}$ on $N_{h-1}$. 

To see that $\sim_{h-1}$ is well-defined, we have to check the above claim that `for some' descendant is equivalent to `for any' descendant. Keeping the same notations, let $x,x'$ be two distinct leaves of $i$, and assume that $B_{xy}=h-1$. Then by the induction hypothesis, $\mathbf F_{h-1}$ satisfies the property $(P_{h-1})$, which implies that $B_{xx'}\leq h-2$. Now consider the triangle $xx'y$ in $\Gamma$. It is is acute isosceles, and $B_{xx'}<B_{xy}$, so $B_{x'y}=B_{xy}=h-1$, which proves our claim. 

The reflexivity and symmetry are clear. For the transitivity, let $i, j, k\in N_{h-1}$  be pairwise distinct such that $i\sim_{h-1} j$ and $j\sim_{h-1} k$. Then if $x\in N_0$ is a leaf that is a descendant of $i$, $y\in N_0$ a leaf descendant of $j$, and $z\in N_0$ a leaf descendant of $k$, we have $B_{ij}=B_{jk}=h-1$, so since the triangle $ijk$ is acute isosceles we have $B_{ik}\leq h-1$. But we also have $B_{ik}\geq h-1$ (otherwise $y$ and $z$ would have a common ancestor of height $\leq h-2$ in $\mathbf F_{h-2}$, which would imply $i=k$), so $B_{ik}=h-1$ and $i\sim_{h-1}k$. 

We define $N_h$ to be the set of equivalence classes for $\sim_{h-1}$ in $N_{h-1}$, and define the edges of $\mathbf F_h$ by keeping those of $\mathbf F_{h-1}$, and adding edges between $N_{h-1}$ and $N_h$ as follows: if $i\in N_{h-1}$ we draw an edge between $i$ and its equivalence class in $N_h$. By construction, $\mathbf F_h$ satisfies the property $(P_h)$. 
\end{itemize}

Since $K-1$ is the maximal edge multiplicity, the forest $\mathbf F_K$ is connected (any two leaves have a common ancestor at height $\leq K-1$), i.e. is actually a tree. By construction, if $\Theta$ is any irregular class such that the part of height $\geq 1$ of its fission tree is equal to $\mathbf F_K$ with the heights of the vertices shifted by 1 (such an irregular class always exists), the corresponding fission diagram $\Gamma_c(\Theta)$ is exactly $\Gamma$, which concludes the proof.
\end{proof}

\begin{example}[cf. Fig. 6 of \cite{boalch2025counting}]
Consider the following graph

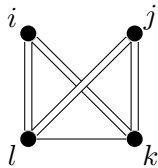
\begin{center}
\begin{tikzpicture}[scale=0.7]
\tikzstyle{vertex}=[circle,fill=black,minimum size=6pt,inner sep=0pt]
\tikzstyle{empty}=[circle,fill=black,minimum size=0pt,inner sep=0pt]
\node[vertex] (i) at (-1,1){};
\draw(-1.3,1.3) node{$i$};
\node[vertex] (j) at (1,1){};
\draw(1.3,1.3) node{$j$};
\node[vertex] (k) at (1,-1){};
\draw(1.3,-1.3) node{$k$};
\node[vertex] (l) at (-1,-1){};
\draw(-1.3,-1.3) node{$l$};
\draw[double distance=0.08cm] (i)--(k);
\draw[double distance=0.08cm] (i)--(l);
\draw[double distance=0.08cm] (j)--(k);
\draw[double distance=0.08cm] (j)--(l);
\draw (k)--(l);
\end{tikzpicture}
\end{center}
It satisfies the ultrametric condition, so it is a fission graph. The corresponding fission tree is

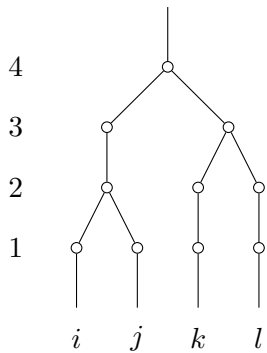
\begin{center}
\begin{tikzpicture}[scale=0.8]
\tikzstyle{inc}=[circle,fill=white,minimum size=4pt,draw, inner sep=0pt]
\tikzstyle{empty}=[circle,fill=black,minimum size=0pt,inner sep=0pt]
\node[empty] (R) at (1.5,5){};
\node[inc] (X) at (1.5,4){};
\node[inc] (ij3) at (0.5,3){};
\node[inc] (kl3) at (2.5,3){};
\node[inc] (ij2) at (0.5,2){};
\node[inc] (k2) at (2,2){};
\node[inc] (l2) at (3,2){};
\node[inc] (i1) at (0,1){};
\node[inc] (j1) at (1,1){};
\node[inc] (k1) at (2,1){};
\node[inc] (l1) at (3,1){};
\node[empty] (i) at (0,0){};
\node[empty] (j) at (1,0){};
\node[empty] (k) at (2,0){};
\node[empty] (l) at (3,0){};
\draw (R)--(X);
\draw (X)--(ij3)--(ij2);
\draw (X)--(kl3);
\draw (ij2)--(i1)--(i);
\draw (ij2)--(j1)--(j);
\draw (kl3)--(k2)--(k1)--(k);
\draw (kl3)--(l2)--(l1)--(l);

\draw (0,-0.5) node {$i$};
\draw (1,-0.5) node {$j$};
\draw (2,-0.5) node {$k$};
\draw (3,-0.5) node {$l$};

\draw (-1,1) node {$1$};
\draw (-1,2) node {$2$};
\draw (-1,3) node {$3$};
\draw (-1,4) node {$4$};
\end{tikzpicture}
\end{center}
An irregular class with this fission tree is for instance $\Theta=I+J+K+L$, with
\[
I=\cir{-z^3-z}, \quad J=\cir{-z^3+z}, \quad K=\cir{z^3-z^2}, \quad L=\cir{z^3+z^2}.
\]
\end{example}

We can use this to obtain our first main result:

\begin{proposition}
\label{prop:counter_example_fission_graph}
There exist nonabelian Hodge graphs which are not fission graphs. 
\end{proposition}

\begin{proof}
Here is an explicit counter-example. Consider the Stokes circles $I=\cir{z^3}$, $J=\cir{z^{4/3}}$, $K=\cir{z^{3/2}}$, and the irregular  class $\Theta=I+J+K$.
A connection whose irregular class is an admissible deformation of $\Theta$ is for instance given by
\[
\nabla=d-Adz, 
\]
where
\[
A={\small\begin{pmatrix}
z^2 & 0 & 0 & 0 & 0& 0\\
0 & 0 & 1 & 0 & 0 & 0\\
0 & 0 & 0 & 1 & 0 & 0\\
0 & z & 0 & 0 & 0 & 0\\
0 & 0 & 0 & 0 & 0 & 1\\
0 & 0 & 0 & 0 & z & 0
\end{pmatrix}}.
\]
One has $B_{I,I}=B_{J,J}=B_{K,K}=0$, so the corresponding nonabelian Hodge diagram is a graph. Furthermore $B_{I,J}=6$, $B_{J,K}=3$, $B_{K,I}=4$, and the graph is the following triangle. 
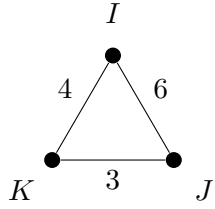
\begin{center}
\begin{tikzpicture}[scale=0.8]
\tikzstyle{vertex}=[circle,fill=black,minimum size=6pt,inner sep=0pt]
\node[vertex] (A) at (0,0){} ;
\node[vertex] (B) at (2,0){} ; 
\node[vertex] (C) at (60:2){};
\draw (A) to node[midway, below] {$3$} (B);
\draw (B) to node[midway, above right] {$6$} (C);
\draw (C) to node[midway, above left] {$4$} (A);
\draw (-0.5, -0.5) node {$K$};
\draw (2.5, -0.5) node {$J$};
\draw (60:2)++(0,0.7) node {$I$};
\end{tikzpicture}
\end{center}
It does not satisfy the ultrametric condition, so it is not a fission graph. From the formula for the dimension involving the Cartan matrix \cite{boalch2020diagrams, doucot2021diagrams}, the dimension of the wild character variety is $22$. 
\end{proof}

\section{Twisted case}
\label{sec:general_case}

In this section, we obtain necessary conditions for being a nonabelian Hodge diagram, and use them to show that there exist graphs that are not nonabelian Hodge graphs and characterise simply-laced nonabelian Hodge graphs. As in the untwisted case, the conditions involve ultrametric inequalities, but things are now more complicated: we have to introduce rescaled versions of edge/loop multiplicities for Stokes circles.

\begin{definition}~
\begin{itemize}
\item 
Let $I$ be a Stokes circle, and $r$ its ramification order. The rescaled loop multiplicity of $I$ is the rational number
\begin{equation}
\widetilde B_{I,I}:=\frac{B_{I,I}-1}{r^2}
\end{equation}
\item Let $I,J$ be two distinct Stokes circles, with respective ramification orders $r,r'$. The rescaled edge multiplicity between $I$ and $J$ is 
\begin{equation}
\widetilde B_{I,J}:=\frac{B_{I,J}}{rr'}.
\end{equation} 
\end{itemize}
\end{definition}

These rescaled quantities are better behaved, and we have:

\begin{theorem}~
\label{thm:necessary_condition_diagram}
\begin{itemize}
\item Let $I,J,K$ be pairwise distinct Stokes circles, and assume that $\widetilde B_{I,J}\leq \widetilde B_{I,K}\leq \widetilde B_{J,K}$. Then 
\begin{equation}
\widetilde B_{I,K}=\widetilde B_{J,K}. 
\end{equation}
\item Let $I,J$ be two distinct Stokes circles. Then
\begin{equation}
\widetilde B_{I,I}\leq \widetilde B_{I,J}. 
\end{equation}
\end{itemize}
\end{theorem}

The first step to prove this is to express the formula for the rescaled edge/loops multiplicities in a convenient form.

\begin{lemma} Keeping the notations of lemmas \ref{lemma:formula_loop_multiplicities} and \ref{lemma:formula_edge_multiplicity}, we have:
\begin{itemize} 
\item Let $I$ be a Stokes circle. The rescaled loop multiplicity is 

\begin{equation}
\widetilde B_{I,I}=k_0(1-\frac{1}{d_0})+ \sum_{j=1}^p  k_j\left(\frac{1}{\mathrm{lcm}(d_0, \dots, d_{j-1})}-\frac{1}{\mathrm{lcm}(d_0, \dots, d_{j})}\right)-1.
\end{equation}
\item Let $I,J$ be two distinct Stokes circles. The rescaled edge multiplicity $\widetilde B_{I,J}$ is given by
\begin{equation}
\widetilde B_{I,J}=k_0(1-\frac{1}{d_0})+ \sum_{j=1}^t k_j\left(\frac{1}{\mathrm{lcm}(d_0, \dots, d_{j-1})}-\frac{1}{\mathrm{lcm}(d_0, \dots, d_{j})}\right)+\frac{f_{I,J}}{\mathrm{lcm}(d_0, \dots, d_{t})}-1.
\end{equation}
\end{itemize}
\end{lemma}

\begin{proof}
For loop multiplicities $\widetilde{B}_{I,I}$, the result follows directly from the fact that for $j=0, \dots, p$ we have
\[
\frac{\mathrm{gcd}(r, m_0, \dots,m_j)}{r}=\frac{1}{\mathrm{lcm}(d_0, \dots, d_j)}.
\]
Let us show this. Take $j\in \{0, \dots,p\}$ 
The smallest possible common denominator for the fractions $k_0,\dots,k_j$ is $r_j:=\mathrm{lcm}(d_0, \dots, d_j)$, i.e. for all $l=0, \dots, j$ we can write
\[
k_l=\frac{n_l}{d_l}=\frac{n_l^{(j)}}{r_j},
\]
which defines an integer $n_l^{(j)}$, and $n_0^{(j)}, \dots, n_j^{(j)}$ are coprime. Now, $r_j$ divides $r=\mathrm{lcm}(d_0, \dots, d_p)$, i.e. there exists an integer $N_j$ such that $r=r_j N_j$. But then we have for $l=1, \dots, j$
\[
k_l=\frac{n_l^{(j)}}{r_j}=\frac{n_l^{(j)}N_j}{r},
\]
which implies $m_l=n_l^{(j)}N_j$ since $m_l$ is defined by $k_l=\frac{m_l}{r}$. Since $\mathrm{gcd}(n_0^{(j)}, \dots, n_j^{(j)})=1$, we thus have $\mathrm{gcd}(m_0,\dots, m_j)=N_j$, and $N_j$ also divides $r$ so $\mathrm{gcd}(r,m_0,\dots, m_j)=N_j=\frac{r}{\mathrm{lcm}(d_0, \dots, d_j)}$, and the desired equality follows. 

For edge multiplicities $\widetilde B_{I,J}$, the result follows from the fact that for $j=0, \dots, t$ we have
\[
\frac{\mathrm{gcd}(r', m'_0, \dots,m'_j)}{r'}=\frac{1}{\mathrm{lcm}(d_0, \dots, d_j)},
\]
which is proven in a completely similar way. 
\end{proof}

\begin{remark}
In particular, notice that $\widetilde B_{I,J}$ only depends on the common part of $I$ and $J$ and of their fission exponent $f_{I,J}$. In terms of the fission tree of the irregular class $\Theta=I+J$, if $i,j$ denote the leaves corresponding to $I$ and $J$ respectively, this means that $\widetilde B_{I,J}$ only depends on the part of the tree which lies above the height of the children of the closest common ancestor $v_{ij}$ (since this height is equal to $f_{I,J})$. 
\end{remark}

\begin{proof}[Proof of Thm. \ref{thm:necessary_condition_diagram}]
Let us first prove the ultrametric condition for rescaled edge multiplicities. Let $I,J,K$ be pairwise distinct Stokes circles. Let us consider the irregular class $\Theta:=I+J+K$, and let $(\mathcal T, \mathbb V, \mathbb A, \mathbb L, h, n)$ be the corresponding fission tree. Let us denote by $i,j,k$ the leaves of the tree corresponding to $I,J,K$ respectively, $f_{ij}:=f_{I,J}$, $f_{jk}:=f_{J,K}$, $f_{ik}:=f_{I,K}$ the fission exponents, $v_{ij}, v_{jk}, v_{ik}$ the common ancestors of pairs of distinct leaves, $h_{ij}, h_{jk}, h_{ik}$ their respective heights, and assume that $h_{ij}\leq h_{ik}\leq h_{jk}$. 

There are two possible cases. First, assume that $h_{ij}=h_{ik}= h_{jk}$. In that case one has $v_{ij}=v_{ik}=v_{jk}$ and $f_{ij}=f_{jk}=f_{ik}$, and the pairs of Stokes circles $(I,J)$, $(I,K)$ and $(J,K)$ have the same common parts.  The situation is represented on the Fig. \ref{fig:proof_ultrametric_first_case}\footnote{here the nonempty vertices of the tree are drawn in white with a dotted border to indicate that they may or may not be mandatory, in agreement with the conventions used in \cite{boalch2025twisted}.}, where $k_0>\dots>k_t$ denote the exponents (if they exist) of the common part (one has $k_t=h_{ij}>f_{ij}$):

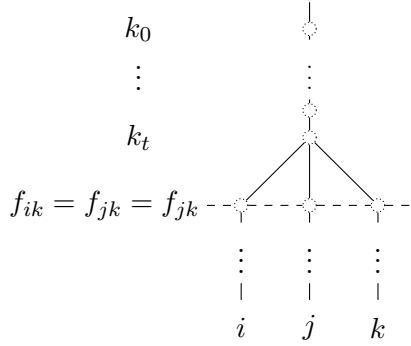
\begin{figure}
\centering
\begin{tikzpicture}[scale=0.9]
\tikzstyle{empty}=[circle,fill=black,minimum size=0pt,inner sep=0pt]
\tikzstyle{mandatory}=[circle,fill=black,minimum size=5pt,draw, inner sep=0pt]
\tikzstyle{indeterminate}=[circle,densely dotted,fill=white,minimum size=5pt,draw, inner sep=0pt]
\node[empty] (R) at (1,4){};
\node[indeterminate] (U1_ijk) at (1,3.6){};
\node (U2_ijk) at (1,3.0){{\scriptsize $\vdots$}};
\node[indeterminate] (U3_ijk) at (1,2.4){};
\node[indeterminate](V_ijk) at (1,2){};
\node[indeterminate] (Y_i) at (0,1){};
\node (Y1_i) at (0,0.3){$\vdots$};
\node[indeterminate] (Y_j) at (1,1){};
\node (Y1_j) at (1,0.3){$\vdots$};
\node[indeterminate] (Y_k) at (2,1){};
\node (Y1_k) at (2,0.3){$\vdots$};
\node[empty](i) at (0,-0.4){};
\node[empty] (j) at (1,-0.4){};
\node[empty] (k) at (2,-0.4){};
\draw (R)--(U1_ijk)--(U2_ijk)--(U3_ijk)--(V_ijk);
\draw (V_ijk)--(Y_i);
\draw (V_ijk)--(Y_j);
\draw (V_ijk)--(Y_k);
\draw (Y_i)--(Y1_i)--(i);
\draw (Y_j)--(Y1_j)--(j);
\draw (Y_k)--(Y1_k)--(k);

\draw[dashed] (-0.5,1)--(Y_i)--(Y_j)--(Y_k)--(2.5, 1);
\draw (-2,1) node {$f_{ij}=f_{ik}=f_{jk}$};
\draw (-1.5,3.6) node {$k_0$};
\draw (-1.5,3) node {$\vdots$};
\draw (-1.5,2) node {$k_t$};

\draw (0,-0.8) node {$i$};
\draw (1,-0.8) node {$j$};
\draw (2,-0.8) node {$k$};
\end{tikzpicture}
\caption{Proof of Thm. \ref{thm:necessary_condition_diagram}, first case.}
\label{fig:proof_ultrametric_first_case}
\end{figure}
This implies that we have exactly the same terms in the expressions for $\widetilde B_{I,J}$, $\widetilde B_{I,K}$, $\widetilde B_{J,K}$ (one term coming from each exponent $k_0,\dots,k_t$ of the common part, and one term for the fission exponent) and thus we obtain $\widetilde B_{I,J}=\widetilde B_{I,K}=\widetilde B_{J,K}$.
 
The second case is when $h_{ij}<h_{jk}$. Then necessarily $v_{ik}=v_{jk}$, the pairs of Stokes circles $(I,K)$ and $(J,K)$ have the same common part, and in the fission tree the situation is as on Fig. \ref{fig:proof_ultrametric_second_case}.

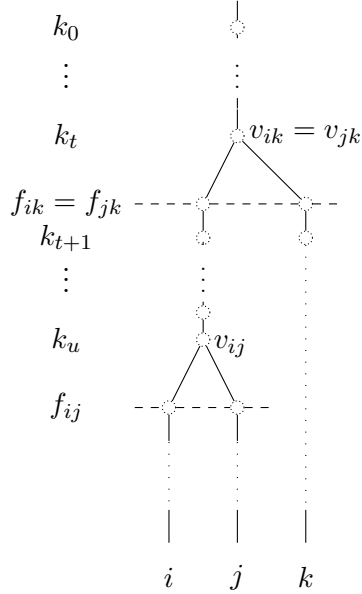
\begin{figure}[h]
\centering
\begin{tikzpicture}[scale=0.9]
\tikzstyle{empty}=[circle,fill=black,minimum size=0pt,inner sep=0pt]
\tikzstyle{mandatory}=[circle,fill=black,minimum size=5pt,draw, inner sep=0pt]
\tikzstyle{indeterminate}=[circle,densely dotted,fill=white,minimum size=5pt,draw, inner sep=0pt]
\node[empty] (R) at (1,8){};
\node[indeterminate] (U1_ijk) at (1,7.6){};
\node (U2_ijk) at (1,7.0){{\scriptsize $\vdots$}};
\node[empty] (U3_ijk) at (1,6.4){};
\node[indeterminate](V_ijk) at (1,6){};
\node[indeterminate] (W_ij) at (0.5,5){};
\node[indeterminate] (W1_ij) at (0.5,4.5){};
\node (W2_ij) at (0.5,4){{\scriptsize $\vdots$}};
\node[indeterminate] (W3_ij) at (0.5,3.4){};

\node[indeterminate] (W_k) at (2,5){};
\node[indeterminate] (W1_k) at (2,4.5){};
\node[empty] (W2_k) at (2,0.5){};

\node[indeterminate] (X_ij) at (0.5,3){};
\node[indeterminate] (Y_i) at (0,2){};
\node[empty] (Y1_i) at (0,1.5){};
\node[empty] (Y2_i) at (0,0.5){};

\node[indeterminate] (Y_j) at (1,2){};
\node[empty] (Y1_j) at (1,1.5){};
\node[empty] (Y2_j) at (1,0.5){};

\node[empty](i) at (0,0){};
\node[empty] (j) at (1,0){};
\node[empty] (k) at (2,0){};
\draw (R)--(U1_ijk)--(U2_ijk)--(U3_ijk)--(V_ijk);
\draw (V_ijk)--(W_ij);
\draw (V_ijk)--(W_k);
\draw (W_ij)--(W1_ij)--(W2_ij)--(W3_ij)--(X_ij);
\draw (W_k)--(W1_k);
\draw[loosely dotted] (W1_k)--(W2_k);
\draw (W2_k)--(k);
\draw (X_ij)--(Y_i);
\draw (X_ij)--(Y_j);

\draw (Y_i)--(Y1_i);
\draw[loosely dotted] (Y1_i)--(Y2_i);
\draw(Y2_i)--(i);

\draw (Y_j)--(Y1_j);
\draw[loosely dotted] (Y1_j)--(Y2_j);
\draw(Y2_j)--(j);

\draw[dashed] (-0.5,2)--(Y_i)--(Y_j)--(1.5, 2);
\draw[dashed] (-0.5,5)--(W_ij)--(W_k)--(2.5, 5);
\draw (-1.5,5) node {$f_{ik}=f_{jk}$};
\draw (-1.5,2) node {$f_{ij}$};
\draw (-1.5,7.6) node {$k_0$};
\draw (-1.5,7) node {$\vdots$};
\draw (-1.5,6) node {$k_t$};
\draw (-1.5, 4.5) node {$k_{t+1}$};
\draw (-1.5,4) node {$\vdots$};
\draw (-1.5,3) node {$k_u$};

\draw (0,-0.5) node {$i$};
\draw (1,-0.5) node {$j$};
\draw (2,-0.5) node {$k$};

\draw (2, 6) node {$v_{ik}=v_{jk}$};
\draw (0.9, 2.9) node {$v_{ij}$};
\end{tikzpicture}
\caption{Proof of Thm. \ref{thm:necessary_condition_diagram}, second case.}
\label{fig:proof_ultrametric_second_case}
\end{figure}

Let us denote by $k_0>\dots > k_t$ the exponents (if they exist) of the common part of $(I,K)$, and $d_0, \dots, d_t$ their respective denominators when written as a reduced fraction. We have $k_t= h(v_{ik})=h(v_{jk})>f_{ik}=f_{jk}$.

Let us also denote by $k_{t+1}>\dots> k_u$ the exponents (if they exist) of the common part of $(I,J)$ which are $\leq f_{ik}$, and by $d_{t+1}, \dots, d_u$ their denominators. We have $k_u=h(v_{ij})>f_{ij}$.

Now the formulas for $\widetilde B_{I,J}, \widetilde B_{I,K}$ and $\widetilde B_{J,K}$ read as follows:
\[
\widetilde B_{I,K}=\widetilde B_{J,K}=k_0(1-\frac{1}{d_0})+ \sum_{j=1}^t  k_j\left(\frac{1}{\mathrm{lcm}(d_0, \dots, d_{j-1})}-\frac{1}{\mathrm{lcm}(d_0, \dots, d_{j})}\right)+\frac{f_{ik}}{\mathrm{lcm}(d_0, \dots, d_{t})}-1.
\]
and
\[
\widetilde B_{I,J}=k_0(1-\frac{1}{d_0})+ \sum_{j=1}^u  k_j\left(\frac{1}{\mathrm{lcm}(d_0, \dots, d_{j-1})}-\frac{1}{\mathrm{lcm}(d_0, \dots, d_{j})}\right)+\frac{f_{ij}}{\mathrm{lcm}(d_0, \dots, d_{u})}-1.
\]
In these two expressions, the terms corresponding to the heights $k_0, \dots, k_t$ are identical. For the remaining terms in $\widetilde B_{I,J}$, since the heights $k_{t+1}, \dots, k_u$ and $f_{ij}$ are smaller than $h_{ik}$, we have
\begin{align*}
&\sum_{j=t+1}^u  k_j\left(\frac{1}{\mathrm{lcm}(d_0, \dots, d_{j-1})}-\frac{1}{\mathrm{lcm}(d_0, \dots, d_{j})}\right)+\frac{f_{ij}}{\mathrm{lcm}(d_0, \dots, d_{u})}\\
\leq & \;f_{ik}\left(\sum_{j=t+1}^u  \left(\frac{1}{\mathrm{lcm}(d_0, \dots, d_{j-1})}-\frac{1}{\mathrm{lcm}(d_0, \dots, d_{j})}\right)+\frac{1}{\mathrm{lcm}(d_0, \dots, d_{u})}\right)\\
=& \;\frac{f_{ik}}{\mathrm{lcm}(d_0, \dots, d_{t})},
\end{align*}
hence $\widetilde B_{I,J}\leq \widetilde B_{I,K}=\widetilde B_{J,K}$.

Let us now prove the inequality involving loop multiplicities. Let $I,J$ be two distinct Stokes circles. Let us denote by $k_0>\dots > k_t$ (if they exist) the exponents of $I$ which are strictly greater than $f_{I,J}$, and by $k_{t+1}>\dots >k_p$ the exponents of $I$ that are less than or equal to $f_{I,J}$. The formulas for $\widetilde B_{I,I}$ and $\widetilde B_{I,J}$ read
\[
\widetilde B_{I,I}=k_0(1-\frac{1}{d_0})+ \sum_{j=1}^p  k_j\left(\frac{1}{\mathrm{lcm}(d_0, \dots, d_{j-1})}-\frac{1}{\mathrm{lcm}(d_0, \dots, d_{j})}\right)-1,
\]
\[
\widetilde B_{I,J}=k_0(1-\frac{1}{d_0})+ \sum_{j=1}^t k_j\left(\frac{1}{\mathrm{lcm}(d_0, \dots, d_{j-1})}-\frac{1}{\mathrm{lcm}(d_0, \dots, d_{j})}\right)+\frac{f_{I,J}}{\mathrm{lcm}(d_0, \dots, d_{t})}-1.
\]
In these two formulas, the terms corresponding to the heights $k_0, \dots, k_t$ are the same. For the remaining terms in $\widetilde B_{I,I}$, since the heights $k_{t+1}, \dots, k_p$ are smaller than $f_{I,J}$ we have
\begin{align*}
&\sum_{j=t+1}^p  k_j\left(\frac{1}{\mathrm{lcm}(d_0, \dots, d_{j-1})}-\frac{1}{\mathrm{lcm}(d_0, \dots, d_{j})}\right)\\
\leq & \;f_{I,J}\left(\sum_{j=t+1}^p \left(\frac{1}{\mathrm{lcm}(d_0, \dots, d_{j-1})}-\frac{1}{\mathrm{lcm}(d_0, \dots, d_{j})}\right)\right)\\
=& \;f_{I,J}\left(\frac{1}{\mathrm{lcm}(d_0, \dots, d_{t})}-\frac{1}{\mathrm{lcm}(d_0, \dots, d_{u})}\right)\\
\leq & \;\frac{f_{I,J}}{\mathrm{lcm}(d_0, \dots, d_{t})},
\end{align*} 
and it follows that $\widetilde B_{I,I}\leq \widetilde B_{I,J}$.
\end{proof}

\begin{example} Let us consider again the irregular classes of Example \ref{ex:some_diagrams}.
\begin{itemize}
\item Consider $I=\cir{z^{5/3}}$, $J=\cir{z^{3/2}}$, $K=\cir{z^{7/3}}$, and $\Theta:=I+J+K$. We have already computed the corresponding matrix $B$ of multiplicities. Using $\mathrm{Ram}(I)=3$, $\mathrm{Ram}(J)=2$, and $\mathrm{Ram}(K)=3$, we find that the matrix of rescaled multiplicities is 
\[
\begingroup
\renewcommand*{\arraystretch}{1.5}
\widetilde B=\begin{pmatrix}
\frac{1}{9} & \frac{2}{3} & \frac{4}{3}\\
\frac{2}{3} & -\frac{1}{4} & \frac{4}{3}\\
\frac{4}{3} & \frac{4}{3} & \frac{5}{9}\\
\end{pmatrix}.
\endgroup
\]
It satisfies the theorem.
\item Consider $I=\cir{z^{5/2}+z^{7/3}}$, $J=\cir{z^{5/2}+z^{5/4}}$, $K=\cir{z^{5/2}}$.
We have already computed the corresponding matrix $B$ of multiplicities. Using $\mathrm{Ram}(I)=6$, $\mathrm{Ram}(J)=4$, and $\mathrm{Ram}(K)=2$, we find that the matrix of rescaled multiplicities is 
\[
\begingroup
\renewcommand*{\arraystretch}{1.5}
\widetilde B=\begin{pmatrix}
\frac{37}{36} & \frac{17}{12} & \frac{17}{12} \\
\frac{17}{12} & \frac{9}{16} & \frac{7}{8} \\
\frac{17}{12} & \frac{7}{8} & \frac{1}{4}
\end{pmatrix}.
\endgroup
\]
It satisfies the theorem.
\end{itemize}
\end{example}

\begin{remark}
Notice that our proof of Theorem \ref{thm:rescaled_nah_diagram_necessary_condition_intro} is quite similar to the proof of Ploski's theorem given by Popescu-Pampu in \cite[§4]{popescu2020ultrametrics} using Eggers-Wall trees.
\end{remark}

As an application, we obtain the answer to our second question about nonabelian Hodge graphs:

\begin{proposition}
\label{prop:counter_example_nah_graph}
There exist graphs that are not nonabelian Hodge graphs.
\end{proposition}

In particular, this implies not all diagrams are nonabelian Hodge diagrams. 

\begin{proof}Consider the following graph $\Gamma=(N,B)$ with vertices $i,j,k$ and edge multiplicities $B_{ij}=B_{ik}=0$, $B_{jk}=1$:
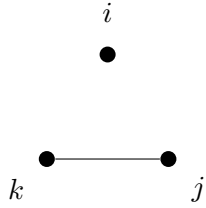
\begin{center}
\begin{tikzpicture}[scale=0.8]
\tikzstyle{vertex}=[circle,fill=black,minimum size=6pt,inner sep=0pt]
\node[vertex] (A) at (0,0){} ;
\node[vertex] (B) at (2,0){} ; 
\node[vertex] (C) at (60:2){};
\draw (A) to node[midway, below] {} (B);
\draw (-0.5, -0.5) node {$k$};
\draw (2.5, -0.5) node {$j$};
\draw (60:2)++(0,0.7) node {$i$};
\end{tikzpicture}
\end{center}

Let us show that it is not a nonabelian Hodge graph. To this end, it is sufficient to show that there does not exist a collection of integers $r=(r_i, r_j, r_k)\in \mathbb Z_{\geq 1}^3$ such that the rescaled diagram $\widetilde{\Gamma}=(N,\widetilde B)$ defined by the decorated diagram $(\Gamma, r)$ has edge multiplicities satisfying the ultrametric conditions. But this is clear: for any $r\in \mathbb Z_{\geq 1}^3$ we have $0=\widetilde B_{ij}=\widetilde B_{ik}<\widetilde B_{jk}$.
\end{proof}

\begin{example}
\label{ex:other_example_not_nah_graph}
Let us give a second example without vanishing edge multiplicities. Consider a graph $\Gamma=(N, B)$ with set of vertices $N=\{a, b, c,d\}$, and where the edge multiplicities $p_{ab}, p_{bc}, p_{cd}, p_{da}, p_{ac}, p_{bd}$ are pairwise distinct prime numbers (for instance as in Fig. \ref{subfig:counter_example_nah_graph_intro}):

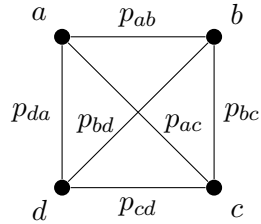
\begin{center}
\begin{tikzpicture}
\tikzstyle{empty}=[circle,fill=black,minimum size=0pt,inner sep=0pt]
\tikzstyle{vertex}=[circle,fill=black,minimum size=6pt,inner sep=0pt]
\node[vertex] (A) at (-1,1){};
\draw (-1.3,1.3) node {$a$};
\node[vertex] (B) at (1,1){};
\draw (1.3,1.3) node {$b$};
\node[vertex] (C) at (1,-1){}; 
\draw (1.3,-1.3) node {$c$};
\node[vertex] (D) at (-1,-1){};
\draw (-1.3,-1.3) node {$d$};
\draw (A)-- node[midway, above]{$p_{ab}$} (B);
\draw (B)-- node[midway, right]{$p_{bc}$} (C);
\draw (C)-- node[midway, below]{$p_{cd}$} (D);
\draw (D)-- node[midway, left]{$p_{da}$} (A);
\draw (A)--  (C);
\draw (B)-- (D);
\draw  (0.6,-0.15) node {$p_{ac}$};
\draw  (-0.55,-0.15) node{$p_{bd}$};
\end{tikzpicture}
\end{center}

Again, to see that this is not a nonabelian Hodge graph, it is sufficient to show that there does not exist a collection of integers $r=(r_a, r_b, r_c, r_d)\in \mathbb Z_{\geq 1}^4$ such that the corresponding rescaled diagram $\widetilde{\Gamma}=(N,\widetilde B)$ satisfies the ultrametric conditions. Reasoning by contradiction, assume that such a collection $r$ exists. 

Consider the triangle $abc$. Without loss of generality, we can assume that 
\[
\widetilde B_{ac}\leq \widetilde B_{ab}=\widetilde B_{bc}.
\]
Since, by definition of rescaling by $r$, $\widetilde B_{ab}=\frac{p_{ab}}{r_a r_b}$, $\widetilde B_{bc}=\frac{p_{bc}}{r_b r_c}$,
this implies
\[
\frac{r_a}{r_c}=\frac{p_{ab}}{p_{bc}}.
\]  
Now consider the triangle $acd$. If $\widetilde B_{ac}\leq \widetilde B_{ad}=\widetilde B_{cd}$, then the same reasoning implies 
\[
\frac{r_a}{r_c}=\frac{p_{da}}{p_{cd}},
\] 
hence $p_{ab}p_{cd}=p_{da}p_{bc}$, and we have a contradiction. So necessarily we have $\widetilde B_{ac}=\widetilde B_{ad}>\widetilde B_{cd}$ or $\widetilde B_{ac}=\widetilde B_{cd}>\widetilde B_{ad}$, say the former. This implies
\[
\frac{r_c}{r_d}=\frac{p_{ac}}{p_{da}}.
\]
Now consider the triangle $bcd$. We have $\widetilde B_{cd}< \widetilde B_{bc}$ (since $\widetilde B_{cd}<\widetilde B_{ac}$, and $\widetilde B_{ac}\leq \widetilde B_{bc}$), so by the ultrametric condition $\widetilde B_{bd}=\widetilde B_{bc}$. This implies
\[
\frac{r_c}{r_d}=\frac{p_{bc}}{p_{bd}},
\]
and we have a contradiction.
\end{example}

As another interesting application, we can characterize the nonabelian Hodge graphs which are simply-laced, i.e. without multiple edges. 

\begin{proposition}
\label{prop:no_new_simply_laced_graphs}
Any simply-laced nonabelian Hodge diagram is a simply-laced fission graph, i.e. a complete $k$-partite graph for some integer $k\geq 1$. 
\end{proposition}

\begin{proof}
Let $\Gamma=(N,B)$ be a simply-laced nonabelian Hodge graph, i.e. such that $B_{i,j}\in \{0,1\}$ for any distinct $i,j\in N$. In the proof of Prop. \ref{prop:counter_example_nah_graph} we have seen that a triangle $ijk$ in $\Gamma$ cannot have edge multiplicities $0,0,1$. The only possibilities for the edge multiplicities in a triangle are thus $\{0,0,0\}$, $\{0,1,1\}$ and $\{1,1,1\}$. This implies that the graph $\Gamma$ satisfies the ultrametric condition, so it is a fission graph.
\end{proof}

\end{document}